\documentclass[10pt, twoside]{amsart}

\usepackage{amsmath, amssymb, amsthm}
\usepackage{enumitem}
\usepackage{hyperref}

\newtheorem{thm}{Theorem}[section]
\newtheorem{cor}[thm]{Corollary}

\newtheorem{pro}[thm]{Proposition}
\newtheorem{lem}[thm]{Lemma}
\theoremstyle{definition}

\def\Aut{{\rm Aut}}
\def\C{{\mathbf C}}
\def\CDS{{\mathcal{CDS}}}

\def\SCDS{{\mathcal {SCDS}}}
\def\dalpha{d_\alpha}

\def\Diff{{\rm Diff}}

\def\H{{\mathbf H}}

\def\Ham{{\rm Ham}}
\def\Hbar{{\overline H}}

\def\id{{\rm id}}
\def\Id{{\rm Id}}

\def\i{\sqrt{-1}\ }
\def\id{{\rm id}}
\def\L{{\mathcal L}}
\def\ma{(M,\alpha)}

\def\PHomeo{{\rm PHomeo}}

\def\R{{\mathbf R}}

\def\re{{\rm ref}}
\def\linspan{{\rm Span}}

\def\nullham{{\mathcal N}}
\def\nullautham{{\mathcal N}_{\rm aut}}
\def\supp{{\rm supp}}
\def\TCDS{{\mathcal {TCDS}}}

\title[Topological contact dynamics III]{Topological contact dynamics III:\\ uniqueness of the topological {H}amiltonian and $C^0$-rigidity of the geodesic flow}
\author[S.~M\"uller \& P.~Spaeth]{Stefan M\"uller \and Peter Spaeth}

\address{Penn State University, Altoona, USA \newline \indent Korea Institute for Advanced Study, Seoul, Korea}
\email{mueller@psu.edu}

\address{Penn State University, Altoona, USA \newline \indent Korea Institute for Advanced Study, Seoul, Korea}
\email{spaeth@psu.edu}


\begin{document}
\thispagestyle{plain}

\begin{abstract}
We prove that a topological contact isotopy uniquely defines a topological contact Hamiltonian.
Combined with previous results from \cite{ms:tcd1}, this generalizes the classical one-to-one correspondence between smooth contact isotopies and their generating smooth contact Hamiltonians and conformal factors to the group of topological contact dynamical systems.
Applications of this generalized correspondence include $C^0$-rigidity of smooth contact Hamiltonians, a transformation law for topological contact dynamical systems, and $C^0$-rigidity of the geodesic flows of Riemannian manifolds.
\end{abstract}

\maketitle

\section{Introduction} \label{sec:intro}
An important characteristic of a Hamiltonian or contact vector field is that the time evolution of the corresponding dynamical system is determined by a single function on the underlying manifold, and conversely this function is unique up to a modest normalization condition.
One goal of this sequence of papers is to extend smooth contact dynamics to topological dynamics, and to generalize the previously stated correspondence to topological  dynamics, so that invariants of topological contact isotopies can be assigned via their uniquely corresponding topological contact Hamiltonians.
As an application we establish $C^0$-contact rigidity (Corollary~\ref{cor:rigidity}).
However, applications of the one-to-one correspondence extend beyond contact and Hamiltonian dynamics.
We recall the extension of the helicity invariant to certain isotopies of measure preserving homeomorphisms of orientable three-manifolds \cite{ms:helicity}.
We also prove a generalized transformation law for topological contact dynamical systems, which provides a new criterion for the topological conjugacy of smooth contact dynamical systems.
See Theorem~\ref{thm:topological-transformation-law} and section~\ref{sec:corollaries} for additional related results.
Finally using the contact-geometric interpretation of the geodesic flow, we prove a novel $C^0$-rigidity phenomenon of the geodesic flow of Riemannian manifolds (Theorem~\ref{thm:geo}).

Let $M$ be a closed oriented smooth manifold of dimension $2n+1$ with contact structure $\xi$, and a contact form $\alpha$ such that $\ker \alpha = \xi$ and $\alpha \wedge (d\alpha)^n$ is a positive volume form.
A topological contact dynamical system $(\Phi, H, h)$ of $(M,\alpha)$~\cite{ms:tcd1, ms:tcd2} arises from a sequence $\Phi_i = \{ \phi_i^t \}$ of smooth contact isotopies of $(M,\xi)$ such that:

\begin{itemize}\itemsep 3pt
\item $\Phi_i$ uniformly converges to an isotopy $\Phi = \{ \phi_t \}$ of homeomorphisms of $M$,
\item the sequence $H_i :[0,1] \times M \to \R$ of smooth time-dependent contact Hamiltonian functions generating the contact isotopies $\Phi_i$ converges with respect to the norm
\begin{gather}\label{eqn:contact-norm}
\qquad \|H_i \| = \int_0^1\!\! \left( \max_{x\in M} H_i(t,x) - \min_{x\in M}H_i(t,x) + \frac{1}{\int_M \nu}\left| \int_M H_i(t,x)\, \nu \right| \right) dt
\end{gather}
to a time-dependent function $H : [0,1] \times M \to \R$, where $\nu$ denotes the canonical measure induced by the volume form $ \alpha \wedge (d\alpha)^n$, and
\item the sequence $h_i$ of smooth time-dependent conformal rescalings of the contact form $\alpha$, i.e.\ the sequence of smooth functions $h_i : [0,1] \times M \to \R$ satisfying $(\phi_i^t)^* \alpha = e^{h^t_i}\alpha$ converges with respect to the uniform norm
\begin{gather}\label{eqn:uniform-norm}
\qquad |h_i| = \max \{ |h_i(t,x)| \mid (t,x) \in [0,1] \times M \}
\end{gather}
to a continuous function $h : [0,1] \times M \to \R$.
\end{itemize}
In our terminology $\Phi$, $H$, and $h$ above such that all three conditions simultaneously hold are called a topological contact isotopy, a topological contact Hamiltonian, and a topological conformal factor, respectively.
See~\cite{ms:tcd2} for examples of non-smooth topological contact dynamical systems of any $\ma$.

Theorem~\ref{thm:main-theorem} is the first result of this paper.
After recalling the necessary preliminaries on contact dynamics in section~\ref{sec:contact-dynamics}, the proof is given in section~\ref{sec:proof-main-theorem}.

\begin{thm}\label{thm:main-theorem}
Let $(\Phi, H, h)$ be a topological contact dynamical system of $(M,\alpha)$.
If $\Phi$ is the constant isotopy at the identity, i.e.\ $\phi_t = \id$ for all $t \in [0,1]$, then $H_t = 0$ for almost every $t \in [0,1]$.
\end{thm}

An obvious feature of contact dynamics that distinguishes itself from Hamiltonian dynamics is the presence of non-trivial conformal rescaling of the contact form.
We addressed the uniqueness of the topological conformal factor in~\cite{ms:tcd1}.

\begin{thm}\label{thm:uniqueness-conformal-factor}
Let $(\Phi, H, h)$ be a topological contact dynamical system of $(M,\alpha)$.
If $\Phi$ is the constant isotopy at the identity, i.e.\ $\phi_t = \id$ for all $t \in [0,1]$, then $h_t = 0$ for all $t \in [0,1]$.
\qed
\end{thm}

Combining Theorems~\ref{thm:main-theorem} and~\ref{thm:uniqueness-conformal-factor} yields the following corollary.
The proof is given at the end of section~\ref{sec:topological-contact-dynamics}.

\begin{cor}[Uniqueness of topological Hamiltonian and conformal factor] \label{cor:unique-hamiltonian-conformal-factor}
Given a contact form $\alpha$ on $M$, a topological contact isotopy $\Phi$ defines a unique topological contact Hamiltonian $H$ and topological conformal factor $h$.
That is, if $(\Phi, H, h)$ and $(\Phi, F, f)$ are two topological contact dynamical systems of $(M,\alpha)$ with the same topological contact isotopy $\Phi$, then $H = F$ and $h = f$. 
\end{cor}

The strategy usually applied to prove a theorem of the type of Theorem~\ref{thm:main-theorem} is to suppose that the conclusion is false, and then derive a contradiction with the corresponding `uniqueness of the isotopy' result.
This is the case for topological Hamiltonians of a symplectic manifold~\cite{buhovsky:ugh13}, topological basic contact Hamiltonians of a regular contact manifold~\cite{banyaga:uch12}, and the proof of Theorem~\ref{thm:main-theorem}.
Using different methods however, Viterbo earlier proved the uniqueness of continuous Hamiltonians of continuous Hamiltonian isotopies of a symplectic manifold~\cite{viterbo:ugh06}.

In~\cite[Theorem~6.4]{ms:tcd1} we proved the uniqueness of the topological contact isotopy of a topological contact Hamiltonian.
\begin{thm}\label{thm:uniqueness-isotopy}
Let $(\Phi, H, h)$ be a topological contact dynamical system of $(M,\alpha)$.
If $H_t = 0$ for almost every $t \in [0,1]$, then $\phi_t = \id$ for all $t \in [0,1]$.
\qed
\end{thm}

Therefore the converse \cite[Theorem~6.4]{ms:tcd1} to Corollary~\ref{cor:unique-hamiltonian-conformal-factor} also holds for topological contact dynamical systems.

\begin{cor}[Uniqueness of topological contact isotopy and conformal factor] \label{cor:unique-isotopy-and-conformal-factor}
Given a contact form $\alpha$ on $M$, a topological contact Hamiltonian $H$ defines a unique topological contact isotopy $\Phi$ with unique topological conformal factor $h$.
That is, if $(\Phi, H, h)$ and $(\Psi, H, f)$ are two topological contact dynamical systems of $(M,\alpha)$ with the same topological contact Hamiltonian $H$, then $\phi_t = \psi_t$ and $h_t = f_t$ for all $t \in [0,1]$. 
\qed
\end{cor}

The previous results can also be interpreted as smooth contact rigidity results.
We prove the first statement~(1) below.
For the proof of statement~(2), see \cite[Corollary~3.4]{ms:tcd2}. 

\begin{cor}[Contact rigidity]\label{cor:rigidity}
Let $(\Phi, H, h)$ be a topological contact dynamical system of $(M,\alpha)$.
\begin{enumerate}
\item \label{item:rigidity-isotopy}
If $\Phi$ is a smooth isotopy of diffeomorphisms, then $H$ and $h$ are smooth functions, $\Phi$ is the smooth contact isotopy generated by the smooth contact Hamiltonian $H$, and $\phi_t^*\alpha = e^{h_t}\alpha$.
\item \label{item:rigidity-hamiltonian}
Conversely if $H$ is a smooth function, then both the isotopy $\Phi$ and function $h$ are smooth, $\Phi$ is the contact isotopy generated by $H$, and $\phi_t^*\alpha = e^{h_t}\alpha$.
\end{enumerate}
In both cases the function $h_t$ is given by 
	\[ h_t = \int_0^t (dH_s(R_\alpha))\circ \phi^s_H\, ds. \]
\end{cor}

In fact by \cite[Lemma~13.1]{ms:tcd1}, Corollary~\ref{cor:rigidity} (1) is equivalent to Corollary~\ref{cor:unique-hamiltonian-conformal-factor}, and Corollary~\ref{cor:rigidity} (2) is equivalent to Corollary~\ref{cor:unique-isotopy-and-conformal-factor}.
As mentioned in \cite{ms:tcd1}, smoothness of the conformal factor $h$ on the other hand does not imply that $\Phi$ or $H$ is smooth; non-smooth strictly contact dynamical systems $(\Phi, H, 0)$ are constructed in~\cite{banyaga:uch12}.

\begin{proof}[Proof of (1)]
By rigidity of contact diffeomorphisms \cite[Theorem~1.3]{ms:tcd1}, the limit $\Phi$ is a smooth contact isotopy, and the limit $h$ of the conformal factors $h_i$ coincides with the smooth conformal factor of the smooth contact isotopy $\Phi$.
By Corollary~\ref{cor:unique-hamiltonian-conformal-factor}, the topological contact Hamiltonian of the limit isotopy is equal to the smooth contact Hamiltonian that generates the isotopy $\Phi$.
\end{proof}

The first part of the paper (sections~\ref{sec:contact-dynamics} and~\ref{sec:proof-main-theorem}) contains background material and the proof of Theorem~\ref{thm:main-theorem}.
Section~\ref{sec:local-uniqueness} provides proofs of local versions of the theorems in the introduction.
The final part of the paper (sections~\ref{sec:corollaries} and~\ref{sec:geodesics}) concerns applications of the uniqueness theorems to smooth and topological contact dynamics, and topological rigidity of geodesic flows of Riemannian manifolds.

\section{Preliminaries on smooth and topological contact dynamics}\label{sec:contact-dynamics}

\subsection{Smooth contact dynamics}\label{sec:smooth-contact-dynamics}

A smooth completely non-integrable hyperplane sub-bundle $\xi \subset TM$ of the tangent bundle $TM$ of a smooth manifold $M$ is called a {\em contact structure} or a {\em contact distribution} on $M$.
A contact structure $\xi$ is locally defined by a smooth differential $1$-form $\alpha$ on $M$, a {\em contact form}, such that $\xi = \ker \alpha$, and the non-integrability condition satisfied by $\xi$ is equivalent to a non-degeneracy condition satisfied by $\alpha$ (locally where $\alpha$ is defined), 
	\[ \alpha \wedge (d\alpha)^n \neq 0, \]
where the necessarily odd dimension of $M$ is $2n+1$.
We assume that $\xi$ is cooriented, so that $\alpha$ is globally defined, and $\alpha \wedge (d\alpha)^n$ is a volume form on $M$.
The choice of contact form $\alpha$ is obviously not unique; any other $1$-form $e^g \alpha$ where $g: M \to \R$ is a smooth function defines the same cooriented contact structure and orientation on $M$, and conversely, if $\beta$ is another smooth $1$-form that defines the same cooriented contact structure $\xi$, then there exists a smooth function $g: M \to \R$ such that $\beta = e^g \alpha$. 

The starting point in smooth contact dynamics is made by fixing the choice of a contact form $\alpha$ defining $\xi$.
This choice determines the important Reeb vector field $R_\alpha$ defined by the equations
	\[ \iota(R_\alpha)d \alpha = 0 \ {\rm and} \ \iota(R_\alpha)\alpha = 1, \]
where $\iota(X)\eta$ denotes the interior product of a differential form $\eta$ with a vector field $X$.
One way to construct symmetries of the contact structure is to begin with a smooth function.
Observe that any vector field $X$ on $M$ can be written as $X = H R_\alpha + Y$ for some function $H:M \to \R$ and section $Y$ of the contact distribution $\xi$, and thus the two equations 
\begin{align}\label{eqn:contact-ham}
	\iota (X) \alpha = H \ \ \mbox{and} \ \ \iota (X) d\alpha = (R_\alpha . H) \alpha - dH
\end{align}
posses a unique solution, which we denote $X_H$.
Here $R_\alpha . H = \iota(R_\alpha) dH$ denotes the derivative of $H$ along $R_\alpha$.
If $\L_{X_H}$ denotes the Lie derivative along the vector field $X_H$, then Cartan's identity yields
	\[ \L_{X_H}\alpha = d (\iota(X_H) \alpha) + \iota(X_H) d\alpha = (R_\alpha . H) \alpha. \]
Thus a vector field $X_H$ defined by~(\ref{eqn:contact-ham}) satisfies $\L_{X_H} \alpha = \mu\, \alpha$ for a smooth function $\mu : M \to \R$, i.e.\ $X_H$ is an infinitesimal automorphism of the contact structure $\xi$, or a {\em contact vector field}, and $H$ is called its {\em contact Hamiltonian function}.
We assume that $M$ is closed, i.e.\ compact without boundary, and for simplicity connected.
Denote by $\Phi_H = \{ \phi^t_H \}$ the solution to the corresponding differential equation
	\[ \frac{d}{dt}\phi_t = X_H\circ \phi_t, \quad \phi_0 = \id ,\]
and observe that for each time $t$,
\begin{align}\label{eqn:contact-flow}
(\phi^t_H)^*\alpha = e^{h_t}\alpha,
\end{align}
where
\begin{align}\label{eqn:conformal-factor}
h_t = \int_0^t (R_\alpha . H) \circ \phi^s_H\, ds. 
\end{align}
Hence at each time $t$, $\phi^t_H$ is a {\em contact diffeomorphism}, i.e.\ a diffeomorphism $\phi$ of $M$ such that $\phi_* \xi = \xi$.
In short, given a choice of contact form $\alpha$, a smooth function $H : M \to \R$ defines a vector field $X_H$ whose smooth flow $\{ \phi^t_H \}$ consists of contact diffeomorphisms for all times $t$.

The preceding is an example of a {\em smooth contact isotopy}, where a smooth isotopy $\Phi = \{ \phi_t \}_{0 \le t \le 1}$ of diffeomorphisms of $M$ is called {\em contact} if each time-$t$ map $\phi_t$ is a contact diffeomorphism, i.e.\ there exists a smooth family of functions $h_t : M \to \R$ on $M$ such that~(\ref{eqn:contact-flow}) holds.
In fact, if we allow time-dependent Hamiltonian functions in the previous construction, then every smooth contact isotopy arises in this way. 
Let $X = \{ X_t \}_{0 \le t \le 1}$ denote the time-dependent smooth vector field generating a smooth contact isotopy $\Phi$ in the sense that
	\[	\frac{d}{dt} \phi_t = X_t \circ \phi_t, \]
and denote by $ H_t $ the smooth time-dependent function $H \colon [0,1] \times M \to \R$ defined by $H_t = \alpha (X_t)$.
An elementary calculation shows that the vector field $X_t$ satisfies $\iota(X_t) d\alpha = \mu_t \alpha - dH_t$, with $\mu_t = R_\alpha . H_t$ so that~(\ref{eqn:contact-ham}) holds, and the functions $\mu_t$ and $h_t$ are related by~(\ref{eqn:conformal-factor}).

A smooth contact isotopy $\Phi$ is called {\em strictly contact} with respect to $\alpha$ if each time-$t$ map $\phi_t$ satisfies $\phi^*_t \alpha = \alpha$.
In this case the conformal factor $h$ vanishes identically, and the generating Hamiltonian $H$ satisfies $R_\alpha . H_t = 0$ for every $t \in [0,1]$.
Such a Hamiltonian is called {\em basic}.
Similarly a contact diffeomorphism $\phi$ is called strictly contact with respect to $\alpha$ if it satisfies $\phi^* \alpha = \alpha$.

Thus the choice of contact form $\alpha$ produces a one-to-one correspondence between smooth contact isotopies with their smooth contact Hamiltonians, and so from this point of view, the choice of contact form may be thought of as the normalization condition in contact dynamics.
We write $\Phi = \Phi_H$ for a smooth contact isotopy generated by the smooth Hamiltonian $H$ and with smooth conformal factor $h$, and such a triple $(\Phi, H, h)$ is what we call a {\em smooth contact dynamical system}.
We denote the group of smooth contact dynamical systems by $\CDS(M,\alpha)$, while the group of contact diffeomorphisms is denoted by $\Diff (M,\xi)$, and $\Diff_0 (M,\xi)$ denotes its identity component.
The group $\SCDS \ma$ of {\em smooth strictly contact dynamical systems} of $\ma$ consists of triples $(\Phi, H, h)$ where $\Phi$ is strictly contact, $H$ is basic, and $h$ vanishes identically. 
The group of strictly contact diffeomorphisms and its identity component are denoted $\Diff \ma$ and $\Diff_0 \ma$, respectively.

For some contact manifolds, the collection of basic functions consists only of functions that depend only on time.
This further distinguishes the study of contact dynamics from Hamiltonian or strictly contact dynamics.

\begin{pro}\label{pro:eberlein}
Let $(B,g)$ be a closed simply-connected Riemannian manifold with strictly negative sectional curvature.
Then the Reeb vector field on the unit cotangent bundle with its canonical contact form has a dense orbit.
\end{pro}
\begin{proof}
By~\cite[Theorem~6.4]{eberlein:gfn73} or~\cite[Theorem~17.6.2 and Theorem~18.3.6]{hasselblatt:imt95} the hypotheses imply the existence of a dense orbit of the geodesic flow on the unit tangent bundle $ST B$.
Under the identification of the unit tangent bundle with the unit cotangent bundle $ST^* B$ via the metric $g$, the corresponding Reeb flow possesses a dense orbit.
\end{proof}
See section~\ref{sec:geodesics} for further details on the identification of the geodesic flow of $ST B$ with the Reeb flow of $ST^* B$.

\begin{pro}\label{pro:dense}
Suppose $(M,\xi)$ is a contact manifold with a contact form $\alpha$ that admits a dense Reeb orbit $\mathcal{O}$.
Then every basic function depends only on time, and $\ma$ admits no strictly contact isotopies other than reparameterizations of the Reeb flow.
In particular, each connected component of $\Diff \ma$ is isomorphic to $\R$, and a strictly contact diffeomorphism $\phi$ is in $\Diff_0 \ma$ if and only if there exists a point $x \in M$ such that $x$ and $\phi (x)$ both lie on $\mathcal{O}$.
\end{pro}
\begin{proof}
A basic function $H_t$ is constant along Reeb orbits, and since $\mathcal{O}$ is dense, $H_t$ must be constant on $M$.
Thus $H$ generates a reparameterization of the Reeb flow, and a diffeomorphism $\phi \in \Diff_0 \ma$ is of the form $\phi_R^s$ for a unique real number $s$, where $\Phi_R$ denotes the Reeb flow of $\ma$.

Suppose that $\phi \in \Diff \ma$, and there exists $x \in M$ such that $x$ and $\phi(x)$ lie on $\mathcal{O}$.
Denote by $s$ the unique real number such that $\phi(x) = \phi^s_R(x)$.
Let $ y \in M$, and choose a sequence $y_k \in \mathcal{O}$ that converges to $y$.
By construction, $y_k = \phi^{s_k}_R(x)$ for some $s_k \in \R$, and
\begin{align*}
\phi(y) & = \lim_{k \to \infty} \phi(y_k) \\
 & = \lim_{k \to \infty}\left( \phi^{s_k}_R \circ \phi \circ (\phi_R^{s_k})^{-1} \right)(y_k) \\
 & = \lim_{k \to \infty} \phi^{s_k}_R \circ \phi(x) = \lim_{k \to \infty} \phi^{s_k}_R \circ \phi^s_R(x)= \lim_{k \to \infty} \phi^s_R \circ \phi^{s_k}_R(x)\\
 & = \lim_{k \to \infty} \phi^s_R(y_k)\\
 & = \phi^s_R(y),
\end{align*}
since strictly contact diffeomorphisms commute with the Reeb flow.
\end{proof}

In some sense manifolds $\ma$ that admit a dense Reeb orbit are the opposite of regular contact manifolds, where every orbit is closed and the group $\Diff_0 \ma = \Ham (M/S^1, \omega)$ is as large as it can be.

\subsection{Topological contact dynamics}\label{sec:topological-contact-dynamics}
The extension of smooth contact dynamics to topological dynamics results from the completion of the group of smooth contact dynamical systems with respect to the {\em contact metric} $d_\alpha$, which encodes the isotopies' topological and dynamical data.
See \cite{mueller:ghh07, mueller:ghc08} for the case of Hamiltonian dynamics of a symplectic manifold, and \cite{banyaga:uch12} for the case of the dynamics of a contact form.

The contact distance between two smooth contact dynamical systems $(\Phi_H, H, h)$ and $(\Phi_F, F,f)$ of $\ma$ is given by
	\[ d_\alpha((\Phi_H, H, h), (\Phi_F, F, f)) = \overline d(\Phi_H, \Phi_F) + | h - f | + \| H - F \|, \]
where $\overline d$ denotes a complete metric that induces the $C^0$-topology on the group of isotopies of homeomorphisms of $M$, and $| \cdot|$ and $\| \cdot \|$ denote the norms in equations~(\ref{eqn:uniform-norm}) and~(\ref{eqn:contact-norm}), respectively.
The contact metric is studied in detail in \cite{ms:tcd1, ms:tcd2}, where the relationship between topological contact and Hamiltonian dynamics is also thoroughly explored.
We caution the reader that the convergence of any two terms in the sequence $(\Phi_{H_i}, H_i, h_i)$ of triples does not imply the convergence of the remaining sequence.
Examples are given in~\cite[section~8]{ms:tcd1}.

Recall that a triple $(\Phi, H, h)$ is called a topological contact dynamical system of $\ma$ if it is the limit with respect to the contact metric $d_\alpha$ of a sequence $(\Phi_{H_i}, H_i, h_i)$ of smooth contact dynamical systems of $\ma$.
Note that the uniform metric on the group of homeomorphisms (or isotopies of homeomorphisms) of $M$ is never complete.
However, a sequence of isotopies of homeomorphisms of $M$ that uniformly converges to an isotopy of {\em homeomorphisms} of $M$ is $C^0$-Cauchy and moreover $C^0$-converges to the same limit.
We showed in~\cite{ms:tcd1} that the collection $\PHomeo(M,\xi)$ of topological contact isotopies of $(M,\xi)$ forms a group, and as the notation suggests does not depend on the choice of contact form $\alpha$ such that $\ker \alpha = \xi$.

The norm~(\ref{eqn:uniform-norm}) on the space of conformal factors is complete in that a Cauchy sequence of smooth conformal factors converges to a continuous time-dependent function on $M$.
We refer to such a limit as a {\em topological conformal factor}.
The contact norm~(\ref{eqn:contact-norm}) is also complete in the following sense.
An equivalence class of Cauchy sequences with respect to the contact norm~(\ref{eqn:contact-norm}) of smooth contact Hamiltonian functions determines a {\em topological Hamiltonian function} $H$, which can be thought of as an element of the space $L^1 ([0,1],C^0 (M))$ of $L^1$-functions of the unit interval taking values in the space $C^0 (M)$ of continuous functions of $M$.
Any two such representatives of the equivalence class are equal almost everywhere in time, and such a representative function $H$ can be defined to be any continuous function at the remaining times $t$ belonging to a set of measure zero.

%
The set $\TCDS \ma$ of topological contact dynamical systems forms a group containing the group of smooth contact dynamical systems as a  subgroup.
In the case of smooth contact isotopies and contact Hamiltonians the following identities are simple consequences of standard techniques for ordinary differential equations.
However in the topological setting more sophisticated techniques are required.
See \cite[section~9]{ms:tcd1} for further details and the proof.
\begin{thm}{\cite[Theorem~6.5]{ms:tcd1}}\label{thm:group-structure}
The set $\TCDS \ma$ admits the structure of a topological group, where for two topological contact dynamical systems $(\Phi_H, H, h)$ and $(\Phi_F, F, f)$ the group operations are given by
\begin{align*}
( \Phi_H, H, h ) \circ (\Phi_F, F, f) = (\Phi_H \circ \Phi_F, H \# F, h \# f) \ {\rm and} \ (\Phi_H, H, h)^{-1} = (\Phi_H^{-1}, \overline H, \overline h),
\end{align*}
where the contact Hamiltonians $H\# F$ and $\overline H$ are given respectively at each time $t$ by
\begin{align*}
(H\# F)_t  & = H_t + (e^{h_t}\cdot F_t ) \circ (\phi^t_H)^{-1} \ {\rm and}\ \overline H_t = -e^{-h_t} ( H_t \circ \phi^t_H).
\end{align*}
The conformal factors $h\# f$ and $\overline h$ are defined at each $t$ by
\begin{align*}
(h \# f)_t & = f_t + h_t \circ (\phi^t_H \circ \phi^t_F) \ {\rm and} \ \overline h_t = -h_t \circ (\phi^t_H)^{-1}. \qed
\end{align*}
\end{thm}

Whereas in smooth contact dynamics the group of canonical transformations or changes of coordinates is given by the group $\Diff(M,\xi)$ of contact diffeomorphisms of $(M,\xi)$, in topological contact dynamics this role is played by the group of {\em topological automorphisms of the contact structure}.
Recall \cite[Definition~6.8 and Theorem~6.9]{ms:tcd1} that a homeomorphism $\phi$ of $M$ is a topological automorphism of the contact structure $\xi$ with unique topological conformal factor $h \in C^0(M)$ if there exists a sequence of contact diffeomorphisms $\phi_j \in \Diff(M,\xi)$ that uniformly converges to $\phi$ and whose smooth conformal factors $h_j$ uniformly converge to the continuous function $h$ on $M$.
The group of topological automorphisms is independent of the choice of contact form $\alpha$ defining $\xi$ \cite[Proposition~6.12]{ms:tcd1}, and will be denoted $\Aut (M,\xi)$.

We proved the following transformation law in \cite{ms:tcd1}.
\begin{thm}{\cite[Theorem~6.13]{ms:tcd1}}\label{thm:transformation-law}
Let $(\Phi_H, H, h)$ be a topological contact dynamical system of $(M,\alpha)$ and $\varphi \in \Aut (M, \xi)$ be a topological automorphism of the contact structure with topological conformal factor $g$.
Then $(\varphi^{-1} \circ \Phi_H \circ \varphi, H^\varphi , h^\varphi)$ is a topological contact dynamical system, where
	\[ (H^\varphi)_t = e^{-g}( H_t \circ \varphi) \ {\rm and} \ (h^\varphi)_t = h\circ \varphi + g - g \circ \varphi^{-1}\circ \phi^t_H \circ \varphi.\qed \]
\end{thm}

In particular suppose that $(\Phi_H, H, h)$ and $(\Phi_F, F, f)$ are topological contact dynamical systems and that $\varphi $ is a topological automorphism of the contact structure with topological conformal factor $g$.
Corollary~\ref{cor:unique-isotopy-and-conformal-factor} implies that if $ H = e^{-g}(F \circ \varphi) $, then $\Phi_H = \varphi^{-1} \circ \Phi_H \circ \varphi$.
By Corollary~\ref{cor:unique-hamiltonian-conformal-factor}, the converse to this statement holds.
See Theorem~\ref{thm:topological-transformation-law}.
Here by a function of the form $F \circ \varphi$, etc.\ we mean the time-dependent function given by $F(t, \varphi(x))$ at $(t,x) \in [0,1] \times M$.
Note that the functions $h\# f$, $\overline h$, and $h^\varphi$ also depend on the isotopies $\Phi_H, \Phi_F$, and $\Phi_H^{-1}$ as indicated above.

Assuming Theorem~\ref{thm:main-theorem}, we are now prepared to give a proof of Corollary~\ref{cor:unique-hamiltonian-conformal-factor}.
\begin{proof}[Proof of Corollary~\ref{cor:unique-hamiltonian-conformal-factor}]
Suppose that $\Phi = \{ \phi_t\}$ is a topological contact isotopy of $(M,\xi)$, and that $(\Phi, H, h)$ and $(\Phi, F, f)$ are both topological contact dynamical systems of $(M,\alpha)$.
The composition law implies that the identity isotopy $\Id = \Phi^{-1}\circ \Phi$ of $M$ has the topological contact Hamiltonian $\Hbar \# F$ given by $ ( \Hbar \# F )_t = e^{-h_t}((F_t - H_t) \circ \phi_t)$, and moreover the topological conformal factor associated to $\Id = \Phi^{-1} \circ \Phi$ is given by the continuous function $-h_t + f_t$.
Theorem~\ref{thm:main-theorem} implies that for almost every $t \in [0,1]$, $F_t - H_t = 0$, or in other words the topological contact Hamiltonians $F$ and $H$ are equal.
Finally by Theorem~\ref{thm:uniqueness-conformal-factor}, for all $t \in [0,1]$, the topological conformal factor satisfies $-h_t + f_t = 0$.  
\end{proof}

\section{The proof of Theorem~\ref{thm:main-theorem}}\label{sec:proof-main-theorem}
Buhovsky and Seyfaddini's use of the space of autonomous null-Hamiltonians in their uniqueness proof for topological Hamiltonians of a symplectic manifold~\cite{buhovsky:ugh13} resembles a proof that translation is continuous in $L^p$ for $1 \leq p < \infty$ (see, e.g.\ \cite[Theorem~8.19]{wheeden:mai77}), and we likewise apply this strategy to prove Theorem~\ref{thm:main-theorem}.
The additional difficulties present for topological contact Hamiltonians include the fact that translations on $\R^{2n+1}$ do not in general preserve its standard contact form, and the existence of non-trivial conformal factors of contact isotopies.

\subsection{Null contact Hamiltonians}
By Theorem~\ref{thm:uniqueness-conformal-factor}, the topological conformal factor $h$ of a topological contact dynamical system of the form $(\Id, H, h)$, where $\Id = \{ \id \}$ denotes the constant isotopy at the identity, satisfies $h = 0$, and thus the following sets are equal
\begin{align*}
\{ H \in L^1 ( & [0,1] , C^0(M)) \mid ( \Id, H, h) \in \TCDS \ma  \} \\
& = \{ H \in L^1 ([0,1] , C^0(M)) \mid ( \Id, H, 0) \in \TCDS \ma  \} .
\end{align*}
Therefore we define the set $\nullham \ma $ of {\em null contact Hamiltonians} of $ \ma $ by
	\[ \nullham \ma  = \{ H \in L^1([0,1],C^0(M)) \mid ( \Id, H, 0) \in \TCDS \ma  \}. \]
Also of interest are the time-independent null contact Hamiltonians defined by
	\[ \nullautham \ma  = \{ H \in \nullham \ma  \mid H_t = H_s \in C^0(M) \ {\rm a.e.}\ s, t \in [0,1] \}. \]
We regard $\nullautham \ma$ as a subspace of $C^0(M)$. 
Then $\nullautham \ma  \subset C^0(M)$ and $\nullham \ma  \subset L^1([0,1], C^0(M))$ are closed in the $C^0$ and $L^1$-topologies, respectively.
The next lemma captures the algebraic properties of the sets $\nullham \ma$ and $\nullautham \ma$, generalizing \cite[Lemma~7]{buhovsky:ugh13} to include automorphisms of the contact structure.
\begin{lem}\label{lem:null-hamiltonians}
The sets $\nullham \ma $ and $ \nullautham \ma $ are closed under addition, taking inverses, and transformation by contact automorphisms $\phi \in \Aut(M, \xi)$.
Both of the sets $\nullham \ma$ and $\nullautham \ma$ are invariant under time reparametrizations, and $\nullautham \ma$ is a vector space over $\R$.
\end{lem}
\begin{proof}
	%
The first three claims follow immediately from the identities in Theorem~\ref{thm:group-structure}.
If $H : [0,1] \times M \to \R$ is a topological Hamiltonian that generates the topological contact isotopy $\Phi_H = \{ \phi_H^t \}$, $a < b$ are real numbers, and $\zeta \colon [a,b] \to [0,1]$ is a smooth function, then the reparameterized isotopy
	\[ \Phi_{H^\zeta} = \{ \phi_{H^{\zeta}}^t \}_{a \le t \le b} = \{ \phi_H^{\zeta (t)} \}_{a \le t \le b} \]
is generated by the topological Hamiltonian $H^\zeta \colon [a,b] \times M \to \R$, defined by
\begin{align}\label{eqn:rep}
	H^\zeta (t, x) = \zeta' (t) \cdot H (\zeta(t), x),
\end{align}
where $\zeta'$ denotes the derivative of $\zeta$.
Since $\phi_{H^\zeta}^t = \phi_H^{\zeta (t)}$, the conformal factor $h^\zeta$ of the isotopy $\Phi_{H^\zeta}$ is given by $h_t^\zeta = h_{\zeta (t)}$.
Thus if $H \in \nullham \ma$ and $\zeta$ is as above,
\begin{align}\label{eqn:reparametrized-null-contact}
(\Id^\zeta, H^\zeta, h^\zeta) = (\Id, H^\zeta, 0) \in \TCDS \ma,
\end{align}
and therefore $H^\zeta \in \nullham \ma$.

In the special case $\zeta (t) = s t$ for a real number $s \in [0,1] $, we also write $H^\zeta = H^s$.
For a time-independent null contact Hamiltonian $H \in \nullautham \ma$ we have
\begin{align}\label{eqn:reparametrized-null-contact-autonomous}
(\Phi_{H^s}, H^s, h^s) = (\Id, H^s, 0) = (\Id, s H, 0),
\end{align}
and $s H \in \nullautham \ma$.
It follows that $\nullautham \ma $ is closed under scalar multiplication.
Indeed when $r \geq 0$, write $r = \sum s_j, 0\leq s_j \leq 1$, and apply our previous remarks, so that $r H \in \nullautham \ma$ if $H \in \nullautham \ma$.
Finally, the case $r < 0$ follows from the fact that $H \in \nullautham \ma$ implies $-H = \Hbar \in \nullautham \ma$. 
\end{proof}

\subsection{Translations and convolutions on the Heisenberg group}
Recall that in general linear translations do not preserve the standard contact form on $\R^{2n+1}$.
Thus we utilize the non-abelian Heisenberg group structure on $\R^{2n+1}$, cf.\ \cite{stein:har93}.
Identifying $\R^{2n+1}$ with $\H^n =  \C^n \times \R$, we write points $x = (x_1, \ldots, x_{2n+1})$ as pairs $x = (x', x_{2n+1})$, where $ x' = (x_1', x_2', \ldots, x_n')$ with $x_j' = x_{2j-1} + \i x_{2j}$ for $j = 1, \ldots, n$.
The group structure on $\H^n$ is defined by
	\[ x \cdot y = (x',x_{2n+1})\cdot (y',y_{2n+1}) =  \left(x' + y', x_{2n+1} + y_{2n+1} + \frac{1}{2} {\rm Im}\, \langle x', y' \rangle \right),\] 
where $\langle x', y' \rangle = \sum_{j=1}^n x'_j \cdot \bar y'_j$ denotes the standard Hermitian inner product on $\C^n$, and ${\rm Im}\, \langle y' , x' \rangle$ denotes its imaginary part.

Let $\tau = (\tau', \tau_{2n+1}) \in \H^n$.
A straightforward calculation shows that the diffeomorphism $R_\tau$ given by right multiplication by $\tau^{-1}$
\begin{align*}
& R_\tau: \H^n \to \H^n \qquad  R_\tau(x) = x \cdot \tau^{-1}
\end{align*}
preserves the contact form
 \[ \alpha_0 =   dx_{2n+1} - \frac{1}{2} \sum_{j=1}^n \left( x_{2j-1}dx_{2j} -  x_{2j} dx_{2j-1} \right),\]
i.e.\ $R_\tau^*\, \alpha_0 = \alpha_0$.
Thus for every $t \in [0,1]$, the map $R_{t \tau}$ is a strictly contact diffeomorphism, and  $\{ R_{t\tau} \}$ defines a strictly contact isotopy.
Left translations are also strictly contact.
In Heisenberg coordinates the basic contact Hamiltonian that generates $\{ R_{t\tau} \}$ is the function
\begin{gather}\label{eqn:heisenberg-translation}
F^\tau : \H^n \to \R \qquad F^\tau(x) = - \tau_{2n+1} - {\rm Im}\, \langle x', \tau' \rangle.
\end{gather}

Denote by $ \nu_0$ the measure induced from the volume form $\alpha_0 \wedge d\alpha_0^n$, which coincides with the usual Lebesgue measure.
The convolution of measurable functions $f$ and $g$ is given by
\begin{align}\label{eqn:convolution}
f * g\, (x) = \int_{\H^n} f(y)\, g(y^{-1}\cdot x)\, \nu_0(y)
\end{align}
whenever~(\ref{eqn:convolution}) is finite.
We will apply the next lemma in the subsequent section.
The proof follows from a straightforward adaptation of the Euclidean case (see for example \cite[Chapter~9]{wheeden:mai77}) to the Heisenberg group.
\begin{lem}\label{lem:convolution}
Let $f  \in L^p(\H^n)$ be compactly supported, and suppose that $K$ is a smooth compactly supported function on $\H^n$ such that $ \int_{\H^n}K\, \nu_0 = 1$.
For $\epsilon > 0$, the functions $f_\epsilon = f * K_\epsilon$ are smooth and compactly supported, where $ K_\epsilon ( x ) = \epsilon^{-(2n+1)}\, K ( x / \epsilon ) $.
If $1 \leq p < \infty $, then $\| f_\epsilon - f \|_p \to 0$ as $\epsilon \to 0$.
When $p = \infty$, $f_\epsilon$ converges to $f$ at every point of continuity of $f$.
In particular, if $f$ is continuous then $f_\epsilon$ converges to $f$ uniformly.
\qed
\end{lem}

\subsection{Uniqueness of time-independent contact Hamiltonians}

The uniqueness of the topological contact isotopy of a topological contact Hamiltonian, Corollary~\ref{cor:unique-isotopy-and-conformal-factor}, enables us to complete the second step in the proof of Theorem~\ref{thm:main-theorem}, namely that the set $\nullautham \ma $ of time-independent null contact Hamiltonians consists only of the zero contact Hamiltonian.
We will need the following simple consequence of Darboux's theorem.

\begin{lem}\label{lem:vanishing-conformal-diff}
For any two points $ x $ and $ y \in M $, there exists a contact diffeomorphism $\psi \in \Diff_0 (M,\xi)$ mapping $x$ to $y$ whose conformal factor vanishes at the point $x$.
If $x, y \in U \subset M$, and $U$ is connected, then the diffeomorphism $\psi$ can be constructed with support contained inside $U$.
\end{lem}
\begin{proof}
It is well known that there exists a diffeomorphism $\phi \in \Diff_0 (M,\xi)$ such that $\phi(x) = y$ and $\phi^* \alpha = e^h \alpha$, where $h : M \to \R$ is a smooth function.
Now consider the contact form $\alpha' = e^{h} \alpha$.
By Darboux's theorem $\alpha$ and $\alpha'$ are diffeomorphic in a neighborhood of $x$, i.e.\ there exists a locally defined contact diffeomorphism $\gamma$ isotopic to the identity such that $\gamma(x) = x$, and $\gamma^* \alpha' = \alpha$.
After appropriately cutting off the smooth contact Hamiltonians that generate $\gamma$ and $\phi$, and extending $\gamma$ by the identity to all of $M$, the composition $\psi = \phi \circ \gamma$ has the desired properties.
\end{proof}

\begin{lem}\label{lem:smooth-step}
A smooth null contact Hamiltonian vanishes identically.
\end{lem}
\begin{proof}
Let $H$ be a smooth null contact Hamiltonian, and denote by $(\Phi_H, H, h)$ the smooth contact dynamical system generated by $H$ as in section~\ref{sec:smooth-contact-dynamics}.
Since $H$ is a null contact Hamiltonian $(\Id, H, 0)$ is a topological contact dynamical system.
By Corollary~\ref{cor:unique-isotopy-and-conformal-factor}, $\Phi_H = \Id$, and thus $H = 0$.
\end{proof}

\begin{lem}\label{lem:autonomous-step}
An autonomous null contact Hamiltonian vanishes identically.
\end{lem}
\begin{proof}
We argue by contradiction that a time-independent null contact Hamiltonian must be locally constant.
Since $M$ is connected, it is then globally constant, and by Lemma~\ref{lem:smooth-step} it must be identically zero.

Suppose $H \in \nullautham \ma $ and there exist points $x \neq y \in U$ such that $H(x) \neq H(y)$, where $U \subset M$ is a Darboux neighborhood with local coordinates $x = (x', x_{2n+1}) \in \H^n$ such that
	\[ \left. \alpha \right|_U = dx_{2n+1} - \frac{1}{2} \sum_{j=1}^n \left( x_{2j-1}dx_{2j} -  x_{2j} dx_{2j-1} \right). \]
Let $\phi \in \Diff_0(M,\xi)$ be such that $\phi(x) = y$, and with conformal factor $g$ such that $g(x) = 0$ as in Lemma~\ref{lem:vanishing-conformal-diff}.
We may further assume that $\phi$ is compactly supported within $U$.
By Lemma~\ref{lem:null-hamiltonians} the function $F = H - e^{-g} (H \circ \phi)$ is a null contact Hamiltonian, which is both compactly supported in $U$ and non-zero.
Consider the set
	\[ \L = \overline{\linspan} \{ F^\phi \mid \phi \in \Diff(M,\xi) \ {\rm and} \ \supp(\phi) \subset U \} \subset C^0(M), \]
where $\phi^*\alpha = e^g \alpha$ and $F^\phi = e^{-g}( F \circ \phi$).
By Lemma~\ref{lem:null-hamiltonians} $\L \subset \nullautham \ma $ is a $C^0$-closed linear subspace.

For $K: \H^n \to \R$ as in Lemma~\ref{lem:convolution} and sufficiently small $\epsilon > 0 $, the convolution
\begin{align*}
F_\epsilon (x) = F * K_\epsilon (x) =  \int_{\H^n} F \left( y \right) K_\epsilon (y^{-1} \cdot x)\, \nu_{0}(y),
\end{align*}
is compactly supported in $U$, where the variables $y = (y',y_{2n+1})$ denote Heisenberg coordinates on $\H^n$ as before.
By Lemma~\ref{lem:convolution}, $F_\epsilon$ is smooth, and $C^0$-converges to $F$ as $\epsilon \to 0$; thus $F_\epsilon$ is non-vanishing for all sufficiently small $\epsilon > 0$.

The function $F_\epsilon$ can be $C^0$-approximated by a Riemann sum of the form (cf.\ \cite[Section~3]{buhovsky:ugh13})
	\[ \sum_{j=1}^N c_j\cdot (R_{\tau_j})^*F \]
with $| \tau_j | < \epsilon$, which a priori is not an element of $\L$, since the right translations $R_{\tau_j}$ are not compactly supported. 
Let $G_j = \rho \cdot F^{\tau_j}$, where $F^{\tau_j}$ satisfies~(\ref{eqn:heisenberg-translation}), and $\rho$ is a smooth cut-off function supported in $U$ that equals $1$ on the set
	\[ W = \{ x \cdot v^{-1} \mid x \in \supp(F), \| v \| < \delta \} \subset \overline W \subset U \]
for $ \delta > 0 $ sufficiently small.
Then $\supp(\phi_{G_j}^t) \subset U$ for all $t \in [0,1]$, and if we further impose that $\delta < \epsilon$, then $\phi_{G_j}^t$ coincides with  $R_{t \tau_j}$ on the support of $F$.
Thus
	\[ \sum_{j=1}^N c_j\cdot (R_{\tau_j})^*F = \sum_{j = 1}^{N} c_j \cdot (\phi^1_{G_j})^*F. \]
The right hand side above is an element of $\L \subset \nullautham \ma$, and this implies $F_\epsilon$ is a smooth non-vanishing null contact Hamiltonian.
This contradicts Lemma~\ref{lem:smooth-step}.
\end{proof}

\subsection{Time-dependent null Hamiltonians}
The last step in proving Theorem~\ref{thm:main-theorem} is to show that for almost every $t \in [0,1]$, the restriction $H_t$ of a null contact Hamiltonian is an element of $\nullautham \ma $.
In the present case of contact Hamiltonians on contact manifolds, the proof requires a minor modification to the proof concerning Hamiltonians of symplectic manifolds from~\cite{buhovsky:ugh13}.

\begin{lem}\label{lem:non-autonomous-step}
Suppose that $H \in \nullham \ma$.
For almost every $t \in [0,1]$, the restricted Hamiltonian $H_t$ satisfies $H_t \in \nullautham \ma$.
\end{lem}
\begin{proof}
Fix a value of $t \in [0,1)$, and define 
	\[ F_j(s, x) = \frac{1}{j} \cdot H \left( t + \frac{s}{j}, x \right) \qquad 0 \leq s \leq 1 \]
for $j$ sufficiently large. 
By Lemma~(\ref{lem:null-hamiltonians}), $F_j$ is a null contact Hamiltonian, and thus $G_j = j \cdot F_j$ is as well.
For the the sequence of null Hamiltonians $G_j$, we have
\begin{align}
\| G_j - H_t \| & < 3 \int_0^1 \max_{x \in M} | G_j(s,x) - H_t(x) | \, ds \label{eqn:integral-inequality} \\  
& = 3 j \int_t^{t + \frac{1}{j}} \max_{x \in M} | H(u,x) - H_t(x) | \, du, \label{eqn:integral-equality}
\end{align}
where~(\ref{eqn:integral-inequality}) follows from~\cite[Lemma~2.4]{ms:tcd1}, which observes that for a function $V:M \to \R$
	\[ \max_{x\in M}(V(x)) - \min_{x\in M}(V(x)) + \frac{1}{\int_M \alpha \wedge (d\alpha)^n} \left| \int_M V \alpha \wedge (d\alpha)^n \right| < 3 \max_{x\in M}|V(x)|,\] 
and~(\ref{eqn:integral-equality}) results from the change of variables $u = t + \frac{s}{j}$.
The Lebesgue differentiation theorem implies that for almost every $t \in [0,1]$
	\[ \lim_{h \to 0^+} \frac{1}{h}\int_t ^{t+h}\max_{x \in M} | H(s,x) - H_t(x) | \, ds = 0, \]
and thus $\| G_j - H_t \| \to 0$ as $j \to \infty$.
Therefore $H_t : M \to \R$ is a null contact Hamiltonian, because $\nullham \ma$ is closed with respect to the norm given by equation~(\ref{eqn:contact-norm}).
Since $H_t$ is time-independent by definition, it follows that $H_t \in \nullautham \ma$.
\end{proof}

\begin{proof}[Proof of Theorem~\ref{thm:main-theorem}]
Suppose that $H \in \nullham \ma$.
Then for almost all $t \in [0,1]$, $H_t \in \nullautham \ma$ by Lemma~\ref{lem:non-autonomous-step}.
By Lemma~\ref{lem:autonomous-step}, $H_t = 0$.
Thus $H = 0$ as an element of the space $ L^1([0,1], C^0(M))$. 
\end{proof}

\section{Local Uniqueness Results}\label{sec:local-uniqueness}

Recall that we proved the local uniqueness of the conformal factor of a topological automorphism of the contact structure in~\cite[Proposition~11.4]{ms:tcd1}.
This is the precise statement.

\begin{pro}[Local uniqueness of topological conformal factor] \label{pro:local-unique-topo-conformal-factor}
Let $U \subset M$ be an open subset of a contact manifold $(M,\xi)$ with contact form $\alpha$ such that $\ker \alpha = \xi$.
Suppose that $\phi_i$ and $\psi_i \in \Diff (M)$ are two sequences of diffeomorphisms such that $\phi_i^* \alpha = e^{h_i} \alpha$ and $\psi_i^* \alpha = e^{g_i} \alpha$ on $U$ where $h_i$ and $g_i$ are smooth functions on $U$.

Suppose that the sequences $\phi_i^{-1} \circ \psi_i$ and $\psi_i^{-1} \circ \phi_i$ converge to the identity uniformly on compact subsets of $U$, and that $h_i$ and $g_i$ converge uniformly on compact subsets of $U$ to continuous functions $h$ and $g$, respectively.
Then $h = g$ on $U$.
\qed
\end{pro}

We comment here that the convergence of the sequences of diffeomorphisms $\phi_i^{-1} \circ \psi_i$ and $\psi_i^{-1} \circ \phi_i$ to the identity does not imply even the pointwise convergence of the functions $h_i$ and $g_i$.
See \cite[Lemma~4.7]{ms:tcd2}.

Each time-$t$ map $\phi^t_H$ of a topological contact dynamical system $(\Phi_H, H, h)$ is a topological automorphism of $\xi$.
Therefore Proposition~\ref{pro:local-unique-topo-conformal-factor} with $U = M$ implies that the continuous function $h_t$ is uniquely determined by the homeomorphism $\phi^t_H$.

\begin{cor}{\cite[Corollary~6.10]{ms:tcd1}} \label{cor:unique-topo-conformal-factor-iso}
The topological conformal factor $h$ of a topological contact isotopy $\Phi$ is uniquely determined by the continuous isotopy $\Phi$ and the contact form $\alpha$.
That is, if $(\Phi, H, h)$ and $(\Phi, F, f)$ are two topological contact dynamical systems with the same topological contact isotopy, then $h = f$.
\qed
\end{cor}

The next result is our local reformulation of Theorem~\ref{thm:uniqueness-isotopy}.
\begin{thm}\label{thm:local-uniqueness-isotopy}
Suppose that $ (\Phi_H, H, h) $ is a topological contact dynamical system of $ \ma $, and let $ U \subset M $ be open.
If for almost every $ t \in [0,1] $ the restriction of the Hamiltonian satisfies $ \left. H_t \right|_U = 0 $, then $ \phi^t_H (x) = x $ for every $t \in [0,1]$ and $ x \in U $.
\end{thm}

In order to prove Theorem~\ref{thm:local-uniqueness-isotopy}, we will need a local version of the contact energy-capacity inequality from \cite[Theorem~1.1]{ms:tcd1}.
\begin{thm}\label{thm:local-contact-energy-capacity-inequality}
Let $(M, \xi)$ be a contact manifold with a contact form $\alpha$ such that $\ker \alpha = \xi$.
Suppose that $U \subset M$ is open and that the time-one map $\phi_H^1 \in \Diff_0(M, \xi)$ of a smooth contact Hamiltonian $H:[0,1]\times M \to \R$ displaces the closure $\overline V$ of an open set $V$ such that $V \subset \overline V \subset U$.
Then there exists a constant $C>0$ independent of the contact isotopy $\Phi_H$, its conformal factor $h:[0,1]\times M \to \R$ given by $(\phi^t_H)^*\alpha = e^{h_t}\alpha$, and its contact Hamiltonian $H$ such that if for all $0 \leq t \leq 1$, $\phi^t_H(\overline V ) \subset U$ then
	\[ 0 < Ce^{-|h|_{\overline V}} \leq \| H_{| {\overline U}}\| \]
where $|h|_{\overline V} = \max \{ |h_t(x)| \mid (t,x) \in [0,1]\times \overline V \}$.
\end{thm}
\begin{proof}
Let $a < b$ be distinct real numbers, and consider the product $\overline V \times [a,b]$ as a subset of the symplectization $M \times \R$ of $M$ with its symplectic structure
	\[ \omega = -d (e^\theta \pi_1^*\alpha ), \]
where $\pi_1:M\times \R \to M$ denotes the projection to the first factor, and $\theta$ represents the coordinate on $\R$.
Recall that the smooth admissible Hamiltonian function $\widehat H:[0,1]\times M \times \R \to \R$ defined by $\widehat H (t, x, \theta) = e^\theta H(t,x)$ generates the admissible smooth Hamiltonian isotopy $\Phi_{\widehat H}$ of the symplectization given by
	\[ \phi^t_{\widehat H}(x, \theta) = \left( \phi^t_H(x), \theta - h_t(x) \right) \]
for all $0 \leq t \leq 1$ and all $(x, \theta) \in M \times \R$.
Hence the time-one map $\phi^1_{\widehat H}$ satisfies 
	\[\phi^1_{\widehat H}(\overline V \times [a,b]) \cap (\overline V \times [a,b]) = \emptyset . \] 

Let $W \subset U$ be an open set such that $\phi^t_H(\overline V) \subset W \subset \overline W \subset U$ for all $0 \leq t \leq 1$, and let $\eta$ be a smooth cut-off function that satisfies $\eta = 1$ on $W$ and $\eta = 0 $ outside $U$.
Set $c = | h |_{\overline V }$, and let $\rho: \R \to [0,1]$ be a smooth cut-off function such that
	\[ \rho(\theta) = \left\{ \begin{array}{l l}
					1 & {\rm if} \ \theta \in [a-c,b+c]\\
					0 & {\rm if} \ \theta \in \R \setminus [a-c-1, b+c+1].
					\end{array}
				\right.
	\]
By construction the time-$t$ maps $\phi^t_{\eta H}$ and $\phi^t_H$ agree on $\overline V \subset M$, and thus if $g$ denotes the conformal factor of the isotopy $\Phi_{\eta H}$, then $g_t(x) = h_t(x)$ for all $0 \leq t \leq 1$ and $x \in \overline V$.

Consider the compactly supported Hamiltonian isotopy $\Phi_{\rho \, \widehat{\eta H}}$ of $M\times \R$ generated by the Hamiltonian $\rho\, \widehat{\eta H}$ defined at each $(t, x, \theta) \in [0,1]\times M \times \R$ by
	\[ \rho\, \widehat{\eta H}(t, x, \theta) = \rho(\theta) e^\theta \eta(x) H(t, x).\]
If $x \in \overline V$ and $\theta \in [a, b]$, then for all $0 \leq t \leq 1$, 
\begin{align*}
\phi^t_{\rho\, \widehat{\eta H}}(x, \theta) = (\phi^t_{\eta H}(x), \theta - h_t(x)) = \phi^t_{\widehat H}(x, \theta).
\end{align*}
In particular
	\[ \phi^1_{\rho\, \widehat{\eta H}}\left( \overline V \times [a,b] \right) \cap \left( \overline V \times [a,b] \right) = \emptyset,\]
and the energy-capacity inequality \cite{lalonde:gse95} implies
\begin{align*}
0 < \frac{1}{2}c(\overline V \times [a,b]) \leq \| \rho\, \widehat{\eta H} \|_{\rm Hofer} \leq e^{b+c+1}\| \eta H \| \leq e^{b+c+1} \| H_{| \overline U} \|,
\end{align*}
where $c(\overline V \times [a,b])$ denotes the Gromov width of $\overline V \times [a,b]$, and
	\[ \| \rho\, \widehat{\eta H} \|_{\rm Hofer} = \int_0^1 \left( \max_{(x, \theta)}(\rho\, \widehat{\eta H_t}) - \min_{(x, \theta)}(\rho\, \widehat{\eta H_t}) \right) dt \]
is the Hofer length of the Hamiltonian isotopy $\Phi_{\rho\, \widehat{\eta H}}$ of $M\times \R$.
Thus
	\[ 0 < \frac{c(\overline V \times [a,b])}{2 e^{  b + 1 } } e^{-| h |_{\overline V}} \leq \| H_{| \overline U} \| \]
and the lemma follows.
\end{proof}

\begin{proof}[Proof of Theorem~\ref{thm:local-uniqueness-isotopy}]
If the conclusion of the theorem is false, then there exists a point $x \in U$, a time $0 < t_0 \le 1$ such that $\phi^{t_0}_H(x) \neq x$, and $\phi^t_H(x) \in U$ for all $0 \leq t \leq t_0$.
By continuity of the map $\phi^{t_0}_H$, there exists an open neighborhood $V$ containing $x$ such that $\phi^{t_0}_H(\overline V) \cap \overline V = \emptyset$, and by shrinking $V$ if necessary, we may further assume that $\phi^t_{H}(\overline V) \subset U$ for all $0 \leq t \leq t_0$.

Let $(\Phi_{H_i}, H_i, h_i)$ be a sequence of smooth contact dynamical systems that converges to $(\Phi_H, H, h)$ in the contact metric $\dalpha$ as $i \to \infty$.
By our assumption and the $C^0$-convergence of the sequence $\Phi_{H_i}$ to $\Phi_H$, for all $i$ sufficiently large, the isotopy $\Phi_{H_i}$ satisfies  $\phi^{t_0}_{H_i}(\overline V) \cap \overline V = \emptyset$, and $\phi^t_{H_i}(\overline V) \subset U$ for all $0 \leq t \leq t_0$.

We reparameterize, and apply Theorem~\ref{thm:local-contact-energy-capacity-inequality} as follows.
For each $i = 1, 2, 3, \ldots$, let $(\Phi_{F_i}, F_i, f_i)$ be the smooth contact dynamical system of $(M,\alpha)$ generated by the smooth contact Hamiltonian $F_i:[0,1]\times M \to \R$ defined by
\begin{align}\label{eqn:reparametrized-hamiltonian}
F_i(s, x) = t_0 H_i(st_0, x).
\end{align}
Then for every $0\leq s \le 1$ and $x \in M$ we have $\phi^s_{F_i}(x) = \phi^{st_0}_{H_i}(x)$ and
\begin{align}\label{eqn:reparametrized-conformal-factor} 
 f_i(s, x) = h_i(st_0, x).
\end{align}
Hence $ \phi^1_{F_i}(\overline V) \cap \overline V = \emptyset \ {\rm and } \ \phi^s_{F_i}(\overline V) \subset U $ for all sufficiently large $i$, and all $0 \leq s \leq 1$.
Theorem~\ref{thm:local-contact-energy-capacity-inequality} implies
\begin{align}\label{eqn:inequality-local-contact-energy}
0 < C e^{-|f_i|_{\overline V}} \leq \| {F_i}_{| U}\|.
\end{align}
Equation~(\ref{eqn:reparametrized-hamiltonian}) implies $\| {F_i}_{| U} \| \leq \| {H_i}_{| U} \|$, and~(\ref{eqn:reparametrized-conformal-factor}) implies that $C e^{-|{h_i}|_{\overline V}} \leq C e^{-|{f_i}|_{\overline V}}$, which combined with inequality~(\ref{eqn:inequality-local-contact-energy}) gives
	\[ 0 < C e^{-|{h_i}|_{\overline V}} \leq \| {H_i}_{| U} \|. \]
The contradiction that results by letting $i \to \infty$ completes the proof.
\end{proof}

\begin{cor}[Local uniqueness of topological contact isotopy and conformal factor]\label{cor:local-uniqueness-isotopy-conformal-factor}
Suppose that $(\Phi, H, h)$ and $(\Psi, F, f)$ are topological contact dynamical systems of $\ma$ and that $U \subset M$ is open.
If for almost every $t \in [0,1]$ the contact Hamiltonian functions satisfy ${H_t}_{| \phi_t( U )} = {F_t}_{| \phi_t(U)}$, then for every $t \in [0,1]$ both ${\phi_t}_{| U} = {\psi_t}_{| U}$ and ${h_t}_{| U} = {f_t}_{| U}$.
\end{cor}
\begin{proof}
By Theorem~\ref{thm:group-structure},
	\[ (\Phi_{G}, G, g) = (\Phi^{-1} \circ \Psi, \overline H\# F,  \overline h \# f)\]
is a topological contact dynamical system, and by assumption for almost every $t \in [0,1]$,
	\[ \left. \left( \overline H \# F \right)_t \right|_U = \left. e^{-h_t}((F_t - H_t)\circ \phi_t)\right|_U = 0.\]
Theorem~\ref{thm:local-uniqueness-isotopy} implies that ${\phi^t_G}_{| U} = \id$ and thus $\left. \psi_t \right|_U = \left. \phi_t \right|_U$.
Finally the fact that $\left. h_t \right|_U = \left. f_t \right|_U$ follows from Proposition~\ref{pro:local-unique-topo-conformal-factor}.
\end{proof}

We next prove the local uniqueness of the topological contact Hamiltonian of a topological contact isotopy, or in other words the generalization of Theorem~\ref{thm:main-theorem} to open subsets of a contact manifold $(M, \xi)$ with contact form $\alpha$ such that $\ker \alpha = \xi$.

\begin{thm}\label{thm:local-uniqueness-hamiltonian}
Suppose that $(\Phi_H, H, h)$ is a topological contact dynamical system of $ \ma $, and there exists an open set $U \subset M$ such that $ \left. \phi^t_H \right|_U = \left. \id \right|_U $ for all $0 \leq t \leq 1$.
Then the restriction of $ H_t  $ to $U$ vanishes for almost every $t \in [0,1]$.
\end{thm}

\begin{proof}
Let $ X = \{ x_0, x_1, \ldots \} $ be a countable and dense subset of $U$.
For each $i = 1, 2, \ldots$ there exists a contact diffeomorphism $\varphi_i \in \Diff (M, \xi)$ with support in $U$ such that $\varphi_i(x_i) = x_0$, and by Lemma~\ref{lem:vanishing-conformal-diff} we may in addition assume that the conformal factor $h_i$ corresponding to $\varphi_i$ vanishes at $x_i$.

For every $i = 1, 2, \ldots$, because $\varphi_i^{-1} \circ \phi^t_{H} \circ \varphi_i = \phi^t_H$, Corollary~\ref{cor:unique-hamiltonian-conformal-factor} implies $H = e^{-h_i} (H \circ \varphi_i) \in L^1( [0,1], C^0(M) )$, and thus there exists a set $S_i \subset [0,1]$ of measure zero such that if $t \notin S_i$, then for every $ x \in M $,
	\[ H_t(x) = e^{-h^t_i(x)}H_t(\varphi_i(x)) .\]
In particular if $t \notin S_i$, then
	\[ H(t, x_i) = e^{h^t_i(x_i)} H(t, \varphi_i(x_i)) = H(t, \varphi_i(x_i)) = H(t, x_0). \]
The union $S = S_1 \cup S_2 \cup \cdots \subset [0,1]$ also has measure zero, and if $t \notin S$, then $H_t \in C^0(M)$ is constant on the dense subset $X \subset U$.
Therefore $H_t$ is constant on $U$, or in other words the restriction of $ H $ to the open set $U$ is equal to an $L^1$-function $F:[0,1] \to \R$.
The triple $(\phi^\chi_R, F, 0)$ is a topological contact dynamical system, where $\phi^\chi_R$ denotes the reparameterization of the Reeb flow at time $\chi (t) = \int_0^t F(s)\, ds$.
By Corollary~\ref{cor:local-uniqueness-isotopy-conformal-factor} the isotopy $\phi_R^{\chi(t)}$ is the identity on $U$.
Thus the reparameterization function $\chi$ is zero, and $F(t) = 0$ for almost every $ t \in [0,1] $.
\end{proof}

\begin{cor}[Local uniqueness of topological Hamiltonian and conformal factor]\label{cor:local-uniqueness-conformal-factor-hamiltonian}
Let $U \subset M$ be an open subset of a smooth contact manifold $(M, \xi)$ with a contact form $\alpha$ such that $\ker \alpha = \xi$.
If $(\Phi, H, h)$ and $(\Psi, F, f) \in \TCDS \ma$ satisfy $\Phi_{| U} = \Psi_{| U}$, then $H_t = F_t$ on the open set $\phi_t(U) = \psi_t(U)$, and ${h_t}_{| U} = {f_t}_{| U}$.
\end{cor}
\begin{proof}
Consider the topological contact dynamical system $(\Phi, G, g)$ as in the proof of Corollary~\ref{cor:local-uniqueness-conformal-factor-hamiltonian}.
\end{proof}

\section{Consequences of the uniqueness of the contact Hamiltonian} \label{sec:corollaries}
In this section we prove a number of consequences of Corollary~\ref{cor:unique-hamiltonian-conformal-factor}.
The proofs are elementary for smooth contact dynamical systems, and extend to topological contact dynamical systems by virtue of the uniqueness of the topological contact Hamiltonian.
By Corollary~\ref{cor:local-uniqueness-conformal-factor-hamiltonian} local versions of the results in this section hold as well.
The converse statements follow from Corollary~\ref{cor:unique-isotopy-and-conformal-factor} and Corollary~\ref{cor:local-uniqueness-isotopy-conformal-factor}.
These results are a good indication of the importance of the one-to-one correspondence established by Corollary~\ref{cor:unique-hamiltonian-conformal-factor} and Corollary~\ref{cor:unique-isotopy-and-conformal-factor}.

Before stating the results, we extend the definition of a contact dynamical system in a straightforward manner to systems that are defined for all times $t \in \R$.
Let $H \colon \R \times M \to \R$ be a smooth function, and $X_H = \{ X_H^t \}$ be the corresponding time-dependent contact vector field.
Since $M$ is closed, the vector field $X_H$ generates a unique isotopy $\Phi_H = \{ \phi_H^t \}$ that is defined for all $t \in \R$, and we call $(\Phi_H, H, h)$ a smooth contact dynamical system defined on $\R$.
Consider the restriction of such a smooth contact dynamical system to a closed interval $[a, b] \subset \R$.
After composition with the contact diffeomorphism $(\phi_H^a)^{-1}$ and a linear reparameterization, we may assume $a = 0$ and $b = 1$, and $\phi_H^0 = \id$, thus reducing the case of a general interval $[a, b]$ to the contact dynamical systems studied in the remainder of this paper.
We call a triple $(\Phi_H, H, h)$ a topological contact dynamical system defined on $\R$ if there exists a sequence of smooth contact dynamical systems $(\Phi_{H_i}, H_i, h_i)$ defined on $\R$, so that the restrictions to each closed interval $[a, b]$ converge with respect to the metric $d_\alpha$ to the restriction of $(\Phi_H, H, h)$ to the same interval $[a, b]$.
Clearly this is equivalent to the convergence $\Phi_{H_i} \to \Phi$, $H_i \to H$, and $h_i \to h$ on compact subsets.
In light of the uniqueness theorems proved in this article and in \cite{ms:tcd1}, this definition is also equivalent to imposing that the restriction of $(\Phi_H, H, h)$ to any closed subset $[a, b]$ is a topological contact dynamical system.

\begin{lem} 
Let $(\Phi_H, H, h)$ be a topological contact dynamical system on $\R$, and suppose that the isotopy $\Phi_H = \{ \phi_H^t \}$ is a one-parameter subgroup of $\Aut (M, \xi)$.
Then the topological contact Hamiltonian $H$ is time-independent, and moreover $H \circ \phi_H^t = e^{h_t} \cdot H$ for all times $t$.
In particular, the energy of the system is preserved, i.e.\ $H \circ \phi_H^t = H$ for all $t$, if and only if the topological conformal factor $h$ vanishes.
\end{lem}

\begin{proof}
Let $s \in \R$.
By hypothesis,
\begin{eqnarray} \label{eqn:commuting}
	\phi_H^t \circ \phi_H^s = \phi_H^{t + s} = \phi_H^s \circ \phi_H^t
\end{eqnarray}
for all $t$, and in particular, $\phi_H^t = \phi_H^{t + s} \circ (\phi_H^s)^{-1}$.
It is straightforward to check that the right-hand side is a topological contact isotopy with topological contact Hamiltonian at time $t$ equal to $H_{t + s}$.
By uniqueness of the topological contact Hamiltonian, Corollary~\ref{cor:unique-hamiltonian-conformal-factor}, $H_{t + s} = H_t$ for all $t$, and therefore $H$ is time-independent.

Similarly, $\phi_H^t = (\phi_H^s)^{-1} \circ \phi_H^t \circ \phi_H^s$, and by the transformation law, Theorem~\ref{thm:transformation-law}, the right-hand side is a topological contact isotopy with topological contact Hamiltonian $e^{- h_s} (H \circ \phi_H^s)$.
Again by Corollary~\ref{cor:unique-hamiltonian-conformal-factor}, $e^{- h_s} \cdot (H \circ \phi_H^s) = H$, proving the second claim.
\end{proof}

Conversely, if the function $H$ is time-independent, and $H \circ \phi_H^t = e^{h_t} \cdot H$ for all times $t$, then the isotopy $\Phi_H = \{ \phi_H^t \}$ is a one-parameter subgroup of $\Aut (M, \xi)$.
See \cite[Lemma~7.7]{ms:tcd1} for the proof.
We point out however that the smooth contact Hamiltonian $H (x, y, z) = z - \sum_{i = 1}^n x_i \cdot y_i$ on $\R^{2 n + 1}$ for instance generates a one-parameter subgroup of $\Diff (M,\xi) \subset \Aut (M,\xi)$ that is not strictly contact.
Cutting off the function $H$ inside a Darboux chart provides a similar example on every contact manifold with arbitrary contact form.
In general, we obtain the following formula for the topological conformal factor of the topological contact isotopy.

\begin{lem}
If a topological contact isotopy $\Phi_H = \{ \phi_H^t \}$ is a one-parameter subgroup of $\Aut (M, \xi)$, then its topological conformal factor satisfies the relation $h_{t + s} = h_t + h_s \circ \phi_H^t = h_s + h_t \circ \phi_H^s$ for all times $s$ and $t$.
\end{lem}

\begin{proof}
The claim follows from equation~(\ref{eqn:commuting}) and the uniqueness of the conformal factor Corollary~\ref{cor:unique-topo-conformal-factor-iso}.
\end{proof}

\begin{thm}\label{thm:topological-transformation-law} 
Suppose $\{ \phi_H^t \}$ and $\{ \phi_F^t \}$ are smooth or topological contact isotopies, and $\phi$ is a topological automorphism of the contact structure $\xi$ with topological conformal factor $g$.
If $\{ \phi_H^t \} = \{ \phi^{-1} \circ \phi_F^t \circ \phi\}$, then $H = e^{-g} (F \circ \phi)$.
\end{thm}

Recall that in our notation this means $H_t = e^{-g} (F_t \circ \phi)$ for all $t$.

\begin{proof}
This follows from the transformation law Theorem~\ref{thm:transformation-law} and from uniqueness of the topological contact Hamiltonian, Corollary~\ref{cor:unique-hamiltonian-conformal-factor}.
\end{proof}

See \cite[Theorem~4.3]{ms:tcd2} for the converse statement.

\begin{lem} 
If a topological contact isotopy $\{ \phi_H^t \}$ commutes with the Reeb flow $\{ \phi_R^s \}$ of the contact form $\alpha$ for all times $s$ and $t$, then its corresponding topological contact Hamiltonian $H$ is basic, i.e.\ $H_t \circ \phi_R^s = H_t$ for all $s$ and $t$.
\end{lem}

\begin{proof}
Take $H = F$ and $\phi = \phi_R^s$ with $g = 0$ in the previous theorem.
\end{proof}

See \cite[Lemma~7.10]{ms:tcd1} for the converse.
Similar arguments establish the following results.

\begin{lem}
Suppose $\phi \in \Aut(M,\xi)$ is a topological automorphism with conformal factor $g$ with respect to $\alpha$, and $\alpha' = e^f \alpha$ is another contact form defining $\xi$.
If $\{ \phi_{R'}^t \} = \{ \phi^{-1} \circ \phi_R^t \circ \phi\}$, then $g = f$.
\end{lem}

\begin{proof}  
Take $H = e^{-f}$ and $F = 1$ in Theorem~\ref{thm:topological-transformation-law}.
\end{proof}

For the converse refer to \cite[Proposition~12.2]{ms:tcd1}.

\begin{lem} 
If a topological automorphism $\phi \in \Aut(M,\xi)$ commutes with the Reeb flow $\{ \phi_R^t \}$ for all $t$, then its topological conformal factor $h$ vanishes identically.
\end{lem}

\begin{proof}
We have $\{ \phi^{-1} \circ \phi_R^t \circ \phi \} = \{\phi_R^t \}$, and thus $e^{-h} \cdot 1 = 1$ by uniqueness of the topological contact Hamiltonian.
\end{proof}

The converse to this lemma can be found in \cite[Lemma~12.3]{ms:tcd1}.

\begin{lem} 
Suppose $\phi$ and $\psi$ are two topological automorphisms of the contact structure $\xi = \ker \alpha$ with topological conformal factors $h$ and $g$, respectively, and
	\[ \phi^{-1} \circ \phi_R^t \circ \phi = \psi^{-1} \circ \phi_R^t \circ \psi, \]
for all $t$, where $\{ \phi_R^t \}$ again denotes the Reeb flow of $\alpha$.
Then $h = g$.
\end{lem}

\begin{proof}
The transformation law Theorem~\ref{thm:transformation-law} together with uniqueness of the topological contact Hamiltonian yields $e^{-h} \cdot 1 = e^{-g} \cdot 1$.
\end{proof}

See \cite[Lemma~12.4]{ms:tcd1} for the converse statement.

\section{Rigidity of the geodesic flow} \label{sec:geodesics}
In this section we prove a rigidity result for the geodesic flow of a Riemannian manifold.
The proof uses the identification of the geodesic flow on the unit tangent bundle with the Reeb flow on the unit cotangent bundle, and the uniqueness of the topological contact isotopy corresponding to a topological Hamiltonian function established by Corollary~\ref{cor:unique-isotopy-and-conformal-factor}.

Let $B$ be a closed smooth manifold with a Riemannian metric $g$.
Recall that $g$ is complete, and there exists a unique smooth vector field $G$ on the tangent bundle $T B$ whose trajectories are of the form $t \mapsto (\gamma (t), \dot \gamma (t)) \in T_{\gamma (t)} B \subset T B$, where $\gamma$ is a geodesic (not necessarily of unit speed) of the Riemannian metric $g$.
The flow of the \emph{geodesic field} $G$ is called the \emph{geodesic flow} of $g$.
The length (with respect to the Riemannian metric $g$) of the tangent vector $\dot \gamma$ is constant along a geodesic $\gamma$, and thus the flow of the geodesic field $G$ preserves the \emph{unit tangent bundle} $S T B$ defined fiber-wise by $S T_b B = \{ v \in T_b B \mid g_b (v,v) = 1 \}$.
In other words, the vector field $G$ is tangent to $S T B$, and the geodesic flow restricts to a geodesic flow on the unit tangent bundle.
The Riemannian metric $g$ gives rise to a bundle isomorphism $\Psi \colon T B \to T^* B$ that is fiber-wise defined by
	\[ \Psi_b \colon T_b B \to T_b^* B, \ \ v \mapsto \iota (v) g_b = g_b (v, \cdot), \]
and the \emph{unit cotangent bundle} $S T^* B$ is by definition the isomorphic image of the unit tangent bundle $S T B$.
The induced bundle metric on $T^* B$ is denoted by $g^*$.

The Liouville one-form $\lambda$ on the cotangent bundle $T^* B$ induces a contact form $\alpha = \lambda_{| ST^*B}$ on the unit cotangent bundle $S T^* B$, where $\lambda_u = u \circ d\pi$ for $\pi \colon T^* B \to B$ the canonical projection.
Its Reeb vector field $R$ is related to the geodesic field $G$ on $S T B$ by the identity $\Psi_* G = R$, and in particular, $\Psi \circ \phi_G^t \circ \Psi^{-1} = \phi_R^t$, or
	\[ (\Psi^{-1} \circ \phi_R^t \circ \Psi) (b,v) = (\gamma (t), \dot \gamma (t)), \]
where $\gamma$ is the unique geodesic (of unit speed) originating in the point $\gamma (0) = b \in B$ with $\dot \gamma (0) = v \in S T_b B$.
See sections 1.4 and 1.5 in \cite{geiges:ict08} for further details.

A Riemannian metric induces a distance function on the manifold $B$.
We assume for the remainder of this section that all distances and the resulting notions of convergence are with respect to a fixed reference metric $g_\re$ on $T B$ and the induced bundle metric $g_\re^*$ on $T^* B$ and distance function on $B$.
We say that a sequence of Riemannian metrics $g_k$ converges to $g$ weakly uniformly if for all pairs of vector fields $X$ and $Y$ on $B$, the functions $g_k (X,Y) \to g (X,Y)$ uniformly on $B$.
If $(g_{i j})$ denotes the coefficient matrix of $g$ in a system of local coordinates, this notion of convergence is equivalent to the uniform convergence of the corresponding coefficient functions $g_{k, i j}$ of the metrics $g_k$ to the functions $g_{i j}$ in every system of local coordinates.

\begin{thm}[Rigidity of geodesic flow] \label{thm:geo}
Let $g$ be a Riemannian metric on a closed manifold $B$, and $g_k$ be a sequence of Riemannian metrics that converges to $g$ weakly uniformly.
Suppose that the geodesic flows $\{ \phi_{G_k}^t \}$ of the metrics $g_k$ are uniformly Cauchy on compact subsets of $\R \times T B$, where $\R$ denotes the time variable.
Then the geodesic flows $\{ \phi_{G_k}^t \}$ converge uniformly on compact subsets of $\R \times T B$ to the geodesic flow $\{ \phi_G^t \}$ of the metric $g$.
\end{thm}

If the sequence $g_k$ of Riemannian metrics $C^1$-converges to $g$, the conclusion of the theorem follows from a standard continuity theorem in the theory of ordinary differential equations. 
That the conclusion of the theorem still holds under the present weaker hypotheses is less obvious, and follows from Corollary~\ref{cor:unique-isotopy-and-conformal-factor}.
The conclusion of the theorem does not hold without the assumption that the geodesic flows $\Phi_{G_k}$ are uniformly Cauchy.
To see this, perturb a Riemannian metric that is flat somewhere in $B$ by a $C^0$-small but $C^1$-large bump. 

\begin{proof}
Without loss of generality we may restrict to the geodesic flows defined for time $0 \leq t \leq 1$ and originating on $S T B$, where $S T B$ denotes the unit tangent bundle of the Riemannian metric $g$.
Denote by $S_k T B$ and $S_k T^* B$ the unit tangent bundle and the unit cotangent bundle of the Riemannian metric $g_k$, respectively, and by $\Psi_k \colon T B \to T^* B$ the bundle isomorphism $v \mapsto \iota (v) g_k$ induced by the metric $g_k$, which restricts to an isomorphism $S_k T B \to S_k T^* B$.
Define a bundle diffeomorphism
	\[ \Phi_k \colon T^* B \setminus B \to T^* B \setminus B, \ \ (b,u) \mapsto \left( b, \frac{ \sqrt{ g^* (b) (u,u)} } { \sqrt{ g_k^* (b) (u,u)} }\cdot u \right), \]
where $g_k^*$ again denotes the bundle metric on $T^* B$ defined by the identification of $T^* B$ with $T B$ via the isomorphism $\Psi_k^{-1}$ and by the metric $g_k$.
Then $\Phi_k$ restricts to a contact diffeomorphism $S T^* B \to S_k T^* B$ with $\Phi_k^* \alpha_k = \left( 1/ \sqrt{ g_k^* } \right) \cdot \alpha$, where the contact form $\alpha_k$ is the restriction of the Liouville form $\lambda$ to $S_k T^* B$.
In local coordinates $q = (q_1, \ldots, q_n)$ on $B$ and $p = (p_1, \ldots, p_n)$ on the fibers of $T^* B$, the Liouville form on $T^* B$ is given by $\lambda = p \, dq = \sum_{i = 1}^n p_i \, dq_i$, and if the Riemannian metric $g$ is in the above local coordinates given by $g_b = \sum g_{i j} (b) \cdot dq_i \otimes dq_j$, then
	\[ g_b^* = g^* (b) = \sum_{i,j = 1}^n g^{i j} (b) \cdot \frac{\partial}{\partial q_i} \otimes \frac{\partial}{\partial q_j}, \]
where $(g^{i j} (b))$ denotes the inverse of the matrix $(g_{i j} (b))$.
Moreover,
	\[ \Phi_k (q,p) = \left( q, \frac{ \sqrt{\sum_{i, j = 1}^n p_i \cdot g^{i j} (q) \cdot p_j}}{\sqrt{\sum_{i, j = 1}^n p_i \cdot g_k^{i j} (q) \cdot p_j}}\cdot p_1, \ldots, \frac{\sqrt{\sum_{i, j = 1}^n p_i \cdot g^{i j} (q) \cdot p_j}}{\sqrt{\sum_{i, j = 1}^n p_i \cdot g_k^{i j} (q) \cdot p_j}}\cdot p_n  \right). \]
If $R_k$ denotes the Reeb vector field of the contact form $\alpha_k$, then $\{ \Phi_k^{-1} \circ \phi_{R_k}^t \circ \Phi_k \}$ is a smooth contact isotopy on $(S T^* B, \alpha)$, generated by the Hamiltonian function $H_k (b,u) = \sqrt{g_k^* (b) (u,u)}$, and $ (\Phi_k^{-1} \circ \phi_{R_k}^t \circ \Phi_k)^* \alpha = e^{h_k} \alpha$, where	
	\[ e^{h_k} = \frac{\sqrt{g_k^* \left( (\Phi_k^{-1} \circ \phi_{R_k}^t \circ \Phi_k) (b,u), (\Phi_k^{-1} \circ \phi_{R_k}^t \circ \Phi_k) (b,u) \right)}}{\sqrt{g_k^* (b) (u,u)}}. \]
By hypothesis, the metrics $g_k \to g$ weakly uniformly, and thus $g_k^* \to g^*$ weakly uniformly.
In particular, the Hamiltonian functions $H_k (b,u) = \sqrt{g_k^* (b) (u,u)} \to \sqrt{g^* (b) (u,u)} = 1$ uniformly on $S T^* B$, and the conformal factors $h_k$ converge to the zero function uniformly.
On the other hand,
	\[ \Phi_k^{-1} \circ \phi_{R_k}^t \circ \Phi_k = (\Phi_k^{-1} \circ \Psi_k) \circ \phi_{G_k}^t \circ (\Phi_k^{-1} \circ \Psi_k)^{-1}, \]
and $\Phi_k^{-1} \circ \Psi_k$ and $(\Phi_k^{-1} \circ \Psi_k)^{-1}$ converge to $\Psi$ and $\Psi^{-1}$, respectively, with respect to the bundle metrics $g_\re$ and $g_\re^*$.
Moreover, $(\phi_{G_k}^t)^{-1} (b,v) = \phi_{G_k}^t (b',- v)$, where $b' = \pi (\phi^t_{G_k}(b,v))$ so that the sequence $\{ \Phi_k^{-1} \circ \phi_{R_k}^t \circ \Phi_k \}$ is in fact $C^0$-Cauchy.
Corollary~\ref{cor:unique-isotopy-and-conformal-factor} implies the $C^0$-convergence of the contact isotopies $\{ \Phi_k^{-1} \circ \phi_{R_k}^t \circ \Phi_k \}$ on $S T^* B$ to the contact isotopy $\{ \phi_R^t \}$ generated by the Reeb vector field $R$ of $\alpha$.
Thus 
\begin{align*}
\Phi_{G_k} & = \Psi^{-1}_{k} \circ \Phi_{R_k} \circ \Psi_k = ( \Psi^{-1}_k \Phi_k) \circ (\Phi_k^{-1} \Phi_{R_k} \Phi_k) \circ (\Phi^{-1}_k \Psi_k) \to \Psi^{-1} \Phi_R \Psi = \Phi_G,
\end{align*}
i.e.\ the geodesic flows $\{ \phi_{G_k}^t \}$ converge in the $C^0$-sense to the geodesic flow $\{ \phi_G^t \}$.
\end{proof}

One can also prove Theorem~\ref{thm:geo} using the description of the geodesic flow as the restriction of a Hamiltonian flow on the cotangent bundle to a sub-level set of a quadratic Hamiltonian function on $T^* B$.
For the proof to go through however one must first generalize the uniqueness theorem for topological Hamiltonian isotopies to quadratic Hamiltonians on cotangent bundles.
Using the local uniqueness of the topological contact isotopy associated to a topological contact Hamiltonian, one can also prove a local rigidity result for Riemannian metrics that converge uniformly on some open subset of $B$, and rigidity of the geodesic flow for complete Riemannian metrics on an open manifold converging uniformly on compact subsets.
Another possible generalization of Theorem~\ref{thm:geo} is to sub-Riemannian structures.

\section*{Acknowledgement}
We would like to thank Kaoru Ono for answering a question of ours concerning the existence of dense Reeb orbits and for pointing out the references in the proof of Proposition~\ref{pro:eberlein}.  

\bibliography{tcd3}
\bibliographystyle{amsalpha}
\end{document}